\documentclass[11pt, a4paper]{amsart}
\usepackage{amsmath}
\newtheorem{theorem}{Theorem}[section]

\newtheorem{lemma}{Lemma}[section]
\newtheorem{proposition}{Proposition}[section]
\newtheorem{corollary}{Corollary}[section]

\usepackage[colorlinks=true,citecolor=cyan,backref=page]{hyperref}

\usepackage{tikz}
\usetikzlibrary{patterns}
\usetikzlibrary{fadings}
\tikzfading[name=fade out, inner color=transparent!0, outer color=transparent!90]
\usepackage[labelfont=normalfont,labelformat=simple]{subcaption}
\usepackage{graphicx}
\usepackage{enumerate}

\title{Quasi-automatic semigroups}

\author{B. Blanchette}
\address{Benjamin Blanchette,
D\'epartement de math\'ematiques, Universit\'e du Qu\'ebec \`a Montr\'eal}
\email{benjamin.blanchette@gmail.com}

\author{C. Choffrut}
\address{Christian Choffrut,
IRIF
Universit\'e Paris-Diderot,
Case 7014
75205 Paris CEDEX 13}
\email{Christian.Choffrut@irif.fr}

\author{C. Reutenauer}
\address{Christophe Reutenauer,
D\'epartement de math\'ematiques, Universit\'e du Qu\'ebec \`a Montr\'eal}
\email{reutenauer.christophe@uqam.ca}

\date{\today}

\begin{document}

\maketitle

\hspace{80pt} \noindent {\it To the memory of Maurice Nivat}

\begin{abstract} A {\em quasi-automatic semigroup} is defined by a finite set of generators, a rational (regular) set of representatives, such that if $a$ is a generator or neutral, then the graph of right multiplication by $a$ on the set of representatives is a rational relation. This class of semigroups contains previously considered semigroups and groups (Sakarovitch, Epstein et al., Campbell et al.). Membership of a semigroup to this class does not depend on the choice of the generators. These semigroups are rationally presented. Representatives may be computed in exponential time. Their word problem is decidable in exponential time. They enjoy a property similar to the so-called Lipschitz property, or fellow traveler property. If graded, they are automatic. In the case of groups, they are finitely presented with an exponential isoperimetric inequality and they are characterized by the weak Lipschitz property.
\end{abstract}

\section{Introduction}

{\em Rational transductions} were one of the subjects of Maurice Nivat's thesis, published in 1968 \cite{N}. They are functions from a free monoid into the set of subsets of another free monoid, whose graph is a rational subset of the product monoid; this graph is then called a {\em rational relation}. See the books of Eilenberg \cite{E}, Berstel \cite{B} and Sakarovitch \cite{Sa2} for further reading.

Rational transductions are very useful tools in many domains. For example, decoding a finite code is a rational transduction which is moreover functional: 
these functions were studied first by Sch\"{u}tzenberger
and are a special case of {\em rational functions} which are 
functions whose graph is a rational  relation.
As another example, rational transductions serve to classify context-free languages, a point of view initiated by Nivat \cite{N}, using the concept of {\em abstract family of languages}; see also Berstel's book \cite{B}.

Discrete group theory is certainly one of the areas
where  rational transductions showed their significance and relevance. 
We claim however that they were not yet employed with their full strength. 
We explain why. 

A Kleene-like result 
shows the equivalence between rational relations and relations 
recognized by  two-tape automata. 
Automatic groups and semigroups, together with their asynchronous versions are defined 
via formal structures whose components are relations recognized by 
special types of two-tapes automata, thus of special rational relations.
The main idea of our contribution
is to substitute arbitrary rational relations for the restricted type of rational 
relations  of the literature.

This leads to the definition of a {\em quasi-automatic semigroup} (or {\em group}). Such a semigroup $S$ has a finite set $A$ of generators, and there exist rational subsets $L\subset A^+$, $R,R_a\subset A^+\times A^+$, $a\in A$, such that, $\mu$ being the canonical homomorphism $A^+\rightarrow S$, one has properties (1), (2) and (3) given in Section \ref{semigroup1}.

We give an brief outline of the contents of our work.

We begin by verifying that the monoid version of this definition is compatible with its semigroup version.
We compare quasi-automatic semigroups to previouly considered classes. We show that they contain strictly the rational semigroups of Sakarovitch. They contain also strictly the automatic semigroups of Campbell et al. They contain the asynchronously automatic semigroups of Wei et al., and we conjecture that this inclusion is strict, although we have no example proving it.

Automatic groups are defined using a set of generators; it is then shown that the definition is independent of the chosen set of generators; the same holds for the asynchronous groups. For semigroups however, automaticity depends on the set of generators (and it is unknown for asynchronous automatic semigroups). We show that our notion of quasi-automatic semigroups is independent of the generators.

We show that one may compute in 
exponential time a representative for each word.  We show that the word problem for a quasi-automatic 
semigroup is decidable in exponential time. We prove a weak Lipschitz property: roughly 
speaking, if two words $u,v$ are at distance at most 1 when viewed in the semigroup, then their prefixes, 
viewed in the semigroup, are at bounded distance; the strong form of this property, which is true for 
automatic semigroups, is when each prefix of $u$ is close to each prefix of $v$ of the same length: for 
groups, this property characterizes automaticity.

We show that if a quasi-automatic semigroup is graded, then it is automatic. Finally, we give two results on quasi-automatic groups. We show that a 
quasi-automatic group is finitely presented and has an exponential isoperimetric inequality: this means 
that the group is the quotient of a free group by a normal subgroup which is finitely generated (as normal 
subgroup) and that each element of it, of reduced length $k$, is a product of at most $C^k$ conjugates of generators of the normal subgroup. The last result is that the weak Lipschitz property for groups implies quasi-automaticity.

Several open questions are given at the end of the article.

A word about the proofs: they are all based on the theory of rational relations (or transductions), as one may find it in Berstel's book. We use 
several times Nivat's {\em bimorphism theorem}\footnote{The third author remembers very well lectures on transductions and in particular on this result, by Maurice Nivat, in 1974 at the University of Paris 7.}, and the {\em composition theorem} of 
Elgot and Mezei (which asserts that the composition of rational transductions is a rational tranduction). 
For complexity matters, we use a construction of Arnold and Latteux 
which is an effective version of the combination of two results: the fact that 
each rational relation contains a rational function with the same domain,
due to Eilenberg, and the fact that a rational is equal to the product of a left and of a right sequential function, due to Elgot and Mezei.

\section{Rationality}\label{rat}
Let $S$ be  semigroup. A subset of $S$ is {\em rational} if it is obtained from finite subsets of $S$ by applying the operations of {\em union} $E\cup F$, {\em product} $EF$ and {\em subsemigroup generation}  
$E^+=\cup_{n\geq 1}E^n$.

If $S$ turns out to be a monoid, then in the previous definition, one may replace subsemigroup generation by {\em submonoid generation} $E^*= 
\cup_{n\geq 0}E^n$. This is because $E^+=EE^*$ and $E^*=\{1\}\cup E^+$.

In the sequel, we are mainly interested in the case where $S$ is a finitely generated free semigoup, a finitely generated free monoid or a direct product of such semigroups.

Rationality of a subset is preserved under direct image by a semigroup homomorphism, see \cite{B} Proposition III.2.2.

A homomorphism from a free monoid to another one is called {\em alphabetic} if it sends each generator onto a generator or onto the empty word.

A rational subset of a free monoid $A^*$ is also called a {\em rational language}, whereas a rational 
subset of a product of two free monoids $A^*\times B^*$ is called a {\em rational relation}, or a {\em 
rational transduction}. The word ``transduction" refers to the fact that a relation $R$, subset of $A^*\times 
B^*$, may be seen as a function from $A^*$ into the set of subsets of $B^*$. More precisely, the {\em 
transduction} $\tau$ associated to a relation $R$ is the function $\tau(u)= \{v|(u,v)\in R\}$. A 
transduction extends naturally to a function from the subsets of $A^*$ to the subsets of $B^*$, which 
preserves arbitrary union.

Each rational subset of $A^+$ is a rational subset of $A^*$, and similarly for rational subsets of $A^+
\times B^+$.

Moreover, the intersection of $A^+$ and of any rational subset of $A^*$ is a rational subset of $A^+$. 
Similarly, 
the intersection of $A^+\times B^+$ and of any rational subset of $A^*\times B^*$ is a rational subset of 
$A^+\times B^+$.

A subset of a semigroup $S$ is called {\em recognizable} if it is the inverse image under some 
homomorphism of $S$ into a finite semigroup of a subset of the latter. Recognizable subsets of $S$ are 
closed under Boolean operations.

By Kleene's theorem, recognizable subsets of a free monoid $A^*$ ($A$ finite) coincide with rational 
subsets. It follows that the set of rational languages is closed under Boolean operations. Another 
consequence is that in a finitely generated semigroup, the intersection of a recognizable subset and of a 
rational subset is rational, see \cite{B} Proposition III.2.6; we apply this in the sequel for the monoid $A^*\times A^*$.

\begin{theorem}[Nivat's bimorphism theorem \cite{N} proposition 4 p. 354; see also \cite{B} Theorem III.
3.2]
For each rational relation $R\subset A^*\times B^*$, there exists a rational subset $H$ of some finitely 
generated free monoid 
$C^*$ and alphabetic homomorphisms $\alpha:C^*\rightarrow A^*$, $\beta:C^*\rightarrow B^*$, such that 
$R=\{(\alpha(w),\beta(w))|w\in H\}$.
One may even assume that for any letter $c\in C$, $\alpha(c)=1$ if and only if $\beta(c)\neq 1$. Thus, for 
any word $w$ in $C^*$, $|w|=|\alpha(w)|+|\beta(w)|$.
\end{theorem}

The {\em inverse} of the relation $R$ is the relation $R^{-1}=\{(u,v)|(v,u)\in R\}$.

The {\em composition} of two relations $R\subset A^*\times A^*$ and $R'\subset A^*\times A^*$ is the 
relation 
$R\circ R'=\{(u,v)\in A^*\times A^*|\exists m\in A^*, (u,m)\in R', (m,v)\in R\}$. Note that we follow the 
conventions of \cite{B}: if $\tau,\tau'$ are the transductions associated to $R,R'$, then the transduction $
\tau\circ \tau'$ is associated to the relation $R\circ R'$.

\begin{theorem} (Elgot and Mezei \cite{EM}; see also \cite{E} Theorem IX.4.1 or \cite{B} Theorem III.4.4)
The composition of two rational relations is a rational relation.
\end{theorem}

Note that there is a canonical monoid embedding $(A\times A)^*\rightarrow A^*\times A^*$. Its image is 
the set of pairs $(u,v)$ such that $|u|=|v|$. Each subset composed of such pairs may be identified with a subset of $(A
\times A)^*$.
We shall use the following result given in \cite{E} Theorem IX.6.1: let $T$ be a rational relation such that if 
$(u,v)\in T$, then $u,v$ have same length (it is called {\em length-preserving}). Then $T$ is rational as subset of the monoid $(A\times A)^*$.

\section{Quasi-automatic semigroups}\label{semigroup}
In this section we introduce a new family of semigroups which we call quasi-automatic and show that
it contains previously defined families such as the rational semigroup
and  the synchronous  and the asynchronous automatic semigroups.
\subsection{Semigroups}\label{semigroup1}
Let $S$ be a semigroup. A {\em quasi-automatic semigroup structure on} 
$S$ is a $5$-tuple $(A,\mu,L,R,(R_a)_{a\in A})$, where $A$ is a finite generating set of $S$,
$\mu$ is the natural semigroup homomorphism $A^+
\rightarrow S$ (sending each $a\in A$ onto itself), where $L\subset A^+$ is a rational language, and where
$R,R_a\subset A^+\times A^+$, for each letter $a\in A$,  are rational relations such that:
\begin{enumerate}
\item $\mu(L)=S$;
\item $R=\{(u,v)\in L\times L|\mu(u)=\mu(v)\}$;
\item for each $a\in A$, $R_a=\{(u,v)\in L\times L|\mu(ua)=\mu(v)\}$.
\end{enumerate}

We say that $(A,\mu,L,R,(R_a)_{a\in A})$ is a quasi-automatic semigroup structure on $S$,  {\em with respect to} $A$.

\subsection{Monoids}
The previous definition of quasi-automatic structure on a semigroup has a natural analogue for monoids. We give it now and show that for 
monoids, the two definitions are equivalent.

Let $M$ be a monoid. A {\em quasi-automatic monoid structure
on} $M$ is is a 5-tuple $(A,\mu,L,R,(R_a)_{a\in A})$, where the finite set $A$ generates $M$ as 
monoid, where 
$\mu$ is the natural monoid homomorphism $A^*\rightarrow M$, where $L$ is
a rational language $\subset A^*$ and where $R, R_a\subset A^*\times A^*$, for each letter $a
\in A$, are rational relations, such that
\begin{enumerate}[(a)]
\item $\mu(L)=M$;
\item
$R=\{(u,v)\in L\times L|\mu(u)=\mu(v)\}$;
\item for each $a\in A$, $R_a=\{(u,v)\in L\times L|\mu(ua)=\mu(v)\}$.
\end{enumerate}

\begin{proposition} Let $M$ be a monoid.

1. If $M$ is generated as semigroup by the finite set $A$ and has a quasi-automatic semigroup structure 
with respect to $A$, then $M$ has a quasi-automatic monoid structure with respect to $A$.

2. If $M$ is generated as monoid by the finite set $A$ and has a quasi-automatic monoid structure with respect to $A$, then $M$ has a 
quasi-automatic semigroup structure with respect to $A$ or $A\cup \{1\}$.
\end{proposition}

\begin{proof}
1. Let $(A,\mu,L,R,(R_a)_{a\in A})$ be a quasi-automatic semigroup structure of $M$. We may extend $\mu$ to $A^*
$ by letting $\mu(1)=1$ and still denote it by $\mu$. Then clearly, (a), (b), (c) are satisfied (details are left to the reader): $(A,\mu,L,R,(R_a)_{a\in A})$ 
is a quasi-automatic monoid structure on $M$.

2. Let $(A,\mu,L,R,(R_a)_{a\in A})$ be a quasi-automatic monoid structure of $M$. Suppose first that $\mu(L
\setminus \{1\})=M$. Then $A$ generates $M$ as semigroup. Moreover, let $\mu'=\mu|A^+$, $L'=L\cap A^+$, $R'=R\cap (A^+\times A^+)$, $R'_a=R_a
\cap (A^+\times A^+)$. Then $\mu',L',R',R'_a$ satisfy (1), (2), (3) (details are left to the reader), so that they define a quasi-automatic semigroup 
structure on $M$.

Suppose on the contrary that $\mu(L\setminus \{1\})\neq M$. This implies that $1\in L$ and that $\mu(L
\setminus \{1\})=M\setminus \{1\}$. We take a new letter $c$ and the new alphabet $B=A\cup \{c\}$. Define the 
semigroup homomorphism  $\mu':B^+\rightarrow M$ by $\mu'(a)=\mu(a)$ if $a\in A$ and $\mu'(c)=1$. 

Let $f:A^*\rightarrow B^*$, $w\mapsto w$ if $w\neq 1$ and $1\mapsto c$. In other words, $f$ is the 
identity on $A^+$ and maps the empty word onto $c$. Denote also by $f$ the mapping $(u,v)\mapsto 
(f(u),f(v))$, $A^*\times A^*\rightarrow B^*\times B^*$.

Let $L'=f(L)$, $R'=f(R)$, $R'_a=f(R_a)$ if $a\in A$, and $R'_c=R'$. We show that $(A\cup\{1\},\mu',L',R',(R'_a)_{a\in A}),R'_c$ is a quasi-automatic semigroup structure on $M$. 

Equalities of points  (1), (2) and (3) of the new structure 
follow from the identity $\mu'\circ f=\mu$. 

It remains to prove that 
$f$ preserves rationality.
Indeed, since rational languages are closed under Boolean operations,
the language $f(L)=L\setminus \{1\}\cup \{c\}$ is rational.
Since $L$ is a rational language, $f(L)=L\setminus \{1\} 
\cup c$ is rational since rational languages are closed under boolean operations. If $T$ is any 
subset of $A^*\times A^*$, one has $f(T)=(T\cup \{(u,c)|(u,1)\in T\} \cup \{(c,v)|(1,v)\in T\} \cup E) 
\setminus (1\times A^*\cup A^*\times 1)$, where $E=(c,c)$ if $(1,1)\in T$ and $E=\emptyset$ otherwise.
If $T$ is rational, so is $f(T)$; indeed, the set 
$\{(u,c)|(u,1)\in T\}$ is rational, since it is $K\times c$, where $K$ is the rational language image of $1$ 
under the rational transduction whose graph is the inverse of $T$; the third set is rational for a similar reason, and $E$ is rational since finite; moreover, the last two sets are recognizable, hence their union too, as well as its 
complement, and rationality is preserved by intersection by a recognizable set.

%
\end{proof}




\begin{theorem} Suppose that a monoid $M$ has a quasi-automatic structure $(A,\mu,L,R,(R_a)_{a\in A})$. Suppose that we know a word $l_1\in L$ such that $\mu(l_1)=1$. Then it is decidable if $M$ is a group.
\end{theorem}

\begin{proof}
We verify first that $a$ is left-invertible if and only if there exists $w\in L$ such that $(w,l_1)$ is in $R_a$. 
Indeed, if this holds, then $\mu(wa)=\mu(l_1)=1$, so that $a$ has the left inverse $\mu(w)$. Conversely, if  
$a$ is left-invertible, then for some $w\in L$, $\mu(wa)=1=\mu(l_1)$ and then $(w,l_1)\in R_a$.

It is decidable to know if there exists $w\in L$ such that $(w,l_1)$ is in $R_a$. Indeed, this is equivalent to
$L\times \{l_1\}\cap R_a\neq \emptyset$. This intersection is rational, in an effective way, since  $L\times 
\{l_1\}$ is recognizable (see \cite{B} Proposition 2.6). Now it is decidable to know if a rational relation is nonempty (see 
\cite{B} Proposition 8.2).

In order to conclude, note that if $M$ is a group, then each $a\in A$ is left invertible. Conversely, if this 
holds then, since $A$ generates $M$ as monoid, each element of $M$ is left invertible; this in turn implies 
that each element in $M$ is invertible; thus $M$ is a group.
\end{proof}

\subsection{Comparison with the rational semigroups of Sakarovitch}

Following \cite{Sa1, PS}, a semigroup $S$ is called {\em rational} if it has a finite generating set $A$ with the 
following properties: there exists a rational language $L\subset A^+$ such that the natural homomorphism
$\mu: A^+\rightarrow M$ induces a bijection $L\rightarrow M$ and that the function $\tau:A^+\rightarrow 
A^+$, $w\mapsto L\cap \mu^{-1}\mu(w)$ is rational (that is, its graph is a rational subset of $A^+\times A^
+$). {\em Rational monoids} are defined similarly, and the two definitions are compatible \cite{PS} p.22.

We show that if $S$ is rational, then $S$ is a quasi-automatic semigroup. Indeed, we have $\mu(L)=S$ by assumption; 
moreover, the relation $R=\{(u,v)\in L\times L|\mu(u)=\mu(v)\}$ is rational, because it is equal to $\{(u,u)|u\in L\}$, which is the image of $L$ under the diagonal homomorphism sending each letter $a$ onto $(a,a)$. 
Moreover, for $a\in A$, $R_a=\{(u,v)\in L\times L|\mu(ua)=\mu(v)\}$ is equal to $\{(u,\tau(ua))|u\in L\}$; this is equal to the intersection of $L\times A^+$ (which is recognizable) with the graph of the rational function sending $u$ onto $\tau(ua)$, which is the composition of $u\rightarrow ua$ (clearly a rational function) followed by $\tau$. Rational functions are closed under composition, by the theorem of Elgot and Mezei, and intersection with a recognizable relation preserves rationality. Hence $R_a$ is rational.

Thus (1), (2), (3) are satisfied. It follows that each rational semigroup is quasi-automatic. The converse is not true, since an infinite group cannot be rational, by \cite{Sa1} Example 4.2; but there are infinite groups that are automatic (for example $\mathbb Z$; see also \cite{ECHLPT} Theorems 3.4.1 and 3.4.5), hence quasi-automatic, as is shown in  Section \ref{autom}.

Note that Mercat introduces in \cite{M} a class of semigroups called {\em strongly automatic}. These are assumed to be embedded in groups.
He shows that strongly automatic monoids are rational (\cite{M} Proposition 3.15), and that the converse does not hold (\cite{M} Exemple 3.17).

\subsection{Comparison with automatic groups and semigroups}\label{autom}

In \cite {ECHLPT} (Definition 2.3.1) are defined {\em (synchronous) automatic groups} and in \cite{CRRT} are defined {\em automatic semigroups} extending 
the first notion to semigroups. Note that automaticity for semigroups depends on the choice of generators, see \cite{CRRT} Example 4.5, although 
it does not depend on the choice of generators for groups, \cite{ECHLPT} Theorem 2.4.1, nor for {\em automatic monoids} in some restrictive sense (the generators must be semigroup generators) \cite{DRR}. 

Furthermore, in \cite{ECHLPT} (Definition 7.2.1) are defined {\em asynchronous automatic groups} (which is independent of the set of 
generators Theorem 7.3.3) and in \cite{WWD} this notion is extended to semigroups (Definition 2.3). The class of asynchronous semigroups is 
strictly larger than the class of automatic semigroups, as follows from the example given in
\cite{ECHLPT} (Example 7.4.1). 

We show that synchronous and asynchronous automatic semigroups are quasi-automatic. For this we must define synchronous and asynchronous automata.

\subsubsection{Synchronous}
Synchronous automata have already been considered in \cite{EES}. Let $\$$ be a new symbol. For $(u,v)$ in $A^*\times A^*$, define $\delta(u,v)=(u\$^i,v\$^j)\in (A\cup \$)^*
$, where the natural integers $i,j$ are chosen to be the smallest possible so that $u\$^i,v\$^j$ have the 
same length. Then an automatic structure is defined by a rational language $L\subset A^+$ such that $\mu(L)=S$; moreover, $R,R_a$ being defined as in Section \ref{semigroup1}, one asks 
that the 
sets $\delta(R)$ and $\delta(R_a)$, which may be identified with subsets of $((A\cup \$)\times(A\cup \$))^*$ (see Section \ref{rat}, last paragraph, for this identification), be rational subsets of this free monoid. 

Since erasing a symbol 
preserves rationality (it is performed by a homomorphism), an automatic structure is also a quasi-
automatic structure. 

It follows that synchronous automatic semigroups are quasi-automatic semigroups.

\subsubsection{Asynchronous}
For an asynchronous automatic structure, one considers two-tape automata on the alphabet $A\cup \{\$\}$, which have a double determinism: 
in any state, the automaton can read only on one of the tapes, 
depending on the state; moreover, for each letter there is at most one transition with this letter. These automata are called {\em deterministic 2-tape automata}; see \cite{RS, FR}.

Relations that are recognized by such automata are rational relations. Indeed, these automata are special cases of transducers, and the latter recognize rational relations, see \cite{B}, Theorem III.6.1. 

Now let $S$ be an asynchronous automatic semigroup. This means that $S$ is
generated by a finite set $A$, and that $\mu, R,R_a$ being defined as before, the relations  
$R\$,R_a\$$ are 
recognized by such automata; here $R\$=\{(u\$,v\$)|(u,v)\in R\}$ and similarly define $R_a\$$. 

In this case, the relations $R\$,R_a\$$ are rational, as seen above. Since $R$ and $R_a$ are the image of $R\$$ and 
$R_a\$$ under the homomorphism defined by the identity on $A$ and $\$\mapsto 1$, they are rational relations.  

It follows that asynchronous automatic semigroups are quasi-automatic.

\section{Properties of quasi-automatic semigroups}

\subsection{Change of representatives}

\begin{proposition}
Let $(A,\mu,L,R,(R_a)_{a\in A})$ be a quasi-automatic semigroup structure on the semigroup $S$. Let $L'\subset L$ be a rational language such that $\mu(L')=S$.Then  $L'$ induces a quasi-automatic structure of $S$.
\end{proposition}

\begin{proof}
The set $L'\times L'$ is a recognizable subset of $A^+\times A^+$, see \cite{B} Theorem III.1.5. It follows that its intersection with $R$ and 
with each $R_a$ is a rational subset of $A^+\times A^+$. This implies the result.
\end{proof}
 
\subsection{Change of generators}
 
\begin{theorem} \label{gen} Let $(A,\mu,L,R,(R_a)_{a\in A})$ be a quasi-automatic structure on the semigroup $S$. If $B$ is another finite set of generators of $S$, then there 
exists a quasi-automatic structure on $S$ with respect to $B$.
\end{theorem} 

The theorem allows us to say that a semigroup is {\em quasi-automatic}: this definition depends only on 
$S$, and not on the chosen 
generating set. 

With the previous notations, let $w\in A^*$. Define $R_w=\{(u,v)\in L\times L|\mu(uw)=\mu(v)\}$. This 
notation is consistent with the notation $R_a$ when $w=a$. Moreover $R_1=R$.

\begin{lemma} $R_w\subset A^+\times A^+$ is a rational relation.
\end{lemma}

\begin{proof} This is clear when $|w|=0$ or $1$. We show that for any words $x,y\in A^+$, $R_{xy}$ is the composition 
of the relations $R_y$ and $R_x$. By the theorem of Elgot and Mezei, this will imply that any $R_w$ is rational, by 
induction on the length of $w$.

Let $(u,v)\in R_{xy}$. There exists by (1) a word $m$ in $L$ such that $\mu(ux)=\mu(m)$. Then $(u,m)\in 
R_{x}$. Moreover $(m,v)\in R_y$, since $\mu(my)=\mu(m)\mu(y)=\mu(ux)\mu(y)=\mu(uxy)=\mu(v)$. Thus 
$(u,m)\in R_x,(m,v)\in R_y$, which implies that $(u,v)\in R_y\circ R_x$.

Conversely, let $(u,v)\in R_y\circ R_x$. There exists $m\in A^+$ such that $(u,m)\in R_x$ and $(m,v)\in R_y$. 
Then $\mu(uxy)=\mu(ux)\mu(y)=\mu(m)\mu(y)=\mu(my)=\mu(v)$, which shows that $(u,v)
\in R_{xy}$.
\end{proof}

\begin{proof} [Proof of Theorem \ref{gen}]
Consider the natural homomorphism $\nu:B^+\rightarrow S$, which is the identity on $B$; it is surjective. 

Each $a\in A$ is a product in $S$ of elements of $B$; we may therefore define a homomorphism $\alpha:A^+\rightarrow B^+$ such that $\alpha(a)=b_1\cdots b_n$, $b_i\in A$, where $a=b_1\cdots b_n$ in $S$. We then have $\mu=\nu\circ \alpha$. 

Define $K=\alpha(L)$. It is a rational language in $B^+$. We have $\nu(K)=\nu\circ\alpha(L)=\mu(L)=S$.

Let $T=\{(u,v)\in K\times K|\nu(u)=\nu(v)\}$. We show that $T=(\alpha\times\alpha)(R)$ (where $(\alpha\times\alpha)(x,y)=(\alpha(x),\alpha(y))$). Let $(u,v)\in T$; 
then $u,v\in K$, hence there exist $x,y\in L$ such that $u=\alpha(x),v=\alpha(y)$; moreover, $\nu(u)=
\nu(v)$, hence $\mu(x)=\nu\alpha(x)=\nu(u)=\nu(v)=\nu\alpha(y)=\mu(y)$ and therefore $(x,y)\in R$, so 
that $(u,v)\in (\alpha\times\alpha)(R)$. Conversely, if $(u,v)\in (\alpha\times\alpha)(R)$, then $u=
\alpha(x),v=\alpha(y)$, $x,y\in L$ and $\mu(x)=\mu(y)$; thus $u,v\in K$ and $\nu(u)=\nu\alpha(x)=\mu(x)=
\mu(y)=\nu\alpha(y)=\nu(v)$, so that $(u,v)\in T$.

Let $b\in B$ and $T_b=\{(u,v)\in K\times K|\nu(ub)=\nu(v)\}$. We show that $T_b=(\alpha\times\alpha)
(R_w)$, where $w\in A^+$ has been chosen in such a way that $\nu (b)=\mu (w)$ ($\mu$ is surjective). Let 
$(u,v)\in T_b$; 
then $u,v\in K$, hence there exist $x,y\in L$ such that $u=\alpha(x),v=\alpha(y)$; moreover, $\nu(ub)=
\nu(v)$, hence $\mu(xw)=\mu(x)\mu(w)=\nu\alpha(x)\nu(b)=\nu(u)\nu(b)=\nu(ub)=\nu(v)=\nu\alpha(y)=
\mu(y)$ and therefore $(x,y)\in R_w$, so that $(u,v)\in (\alpha\times\alpha)(R_w)$. Conversely, if $(u,v)\in 
(\alpha\times\alpha)(R_w)$, then $u=
\alpha(x),v=\alpha(y)$, $x,y\in L$ and $\mu(xw)=\mu(y)$; thus $u,v\in K$ and $\nu(ub)=\nu\alpha(x)
\nu(b)=
\mu(x)\mu(w)=\mu(xw)=
\mu(y)=\nu\alpha(y)=\nu(v)$, so that $(u,v)\in T_b$.

Since $\alpha\times\alpha$ is a homomorphism, it preserves rationality and $T,T_b$ are therefore rational. Thus (1), (2) and (3) are proved and there exists a quasi-automatic structure on $S$ with respect to $B$.
\end{proof}

\subsection{Computing representatives}

The term ``representative''  suggests 
the choice of a unique element in an equivalence class.
This is not quite the meaning here. The idea is, given
an arbitrary word  $u\in A^{*}$  to  associate 
a word  $v\in L$ 
with the same image:  $\mu(u)=\mu(v)$. If $\mu$ does not map $L$
on $\mu(L)$ bijectively, uniqueness of such a word $v$ is not guaranteed. 
Thus, in our context, since the computations in the semigroup
(monoid or group) are done via $L$,  by  ``representative'' of an arbitrary word, 
we mean a word in
$L$ that has the same image by $\mu$.

\begin{theorem}\label{repr}
Let $(A,\mu,L,R,(R_a)_{a\in A})$ be a quasi-automatic semigroup structure on the semigroup $S$. There exists a function $l:A^+
\rightarrow L$ such that for any word $u\in A^+$ and any letter $a\in A$, $\mu(l(u))=\mu(u)$,  $\mu(l(u)a)=
\mu(l(ua))$ and 
$(l(u),l(ua))\in R_a$. Moreover, for some $N>1$, the length of $l(u)$ is $\leq N^{|u|}$ and $l(u)$ may be computed in exponential time with respect to the length of $u$.
\end{theorem}

\begin{proof}
Let $\tau_a$ be a rational function $A^*\rightarrow A^*$ such that $(u,\tau_a(u)) \in R_a$ for any word $u$. Such a function exists by Eilenberg's cross-section theorem, \cite{E} Proposition IX.8.2. By a theorem of Elgot and Mezei, each rational function is the 
product of a left and of a right subsequential function (see also \cite{B} Theorem 5.2). Note that one may 
use also the theorem in \cite{AL}, that shows directly that each rational transduction contains a function, 
with the same domain, and which is the composition of a left and of a right sequential function. 
Furthermore, this result is effective in the sense that it actually constructs the two sequential functions.

The image of a word $u$ by a sequential function may be computed in linear time in $|u|$, and its length 
is not more than linear in $|u|$. This follows since such a function is computed by a deterministic automata with 
output.

Thus we may find $N>1, C>0$ such that $\forall a\in A, \forall u\in A^+$, $|\tau_a(u)|\leq N|u|$ and the 
computing time of $\tau_a(u)$ is $\leq C|u|$.

There exist words $l(a)\in L,a\in A$, such that $\mu(a)=\mu(l(a))$, and we may assume that $|l(a)|\leq N$.

We define $l(u)=\tau_{a_1}\tau_{a_2}\cdots\tau_{a_{n-1}}(l(a_n))$ for any word $u=a_n\cdots a_2a_1$, 
$a_i\in A$. 

By construction, we have $l(ua)=\tau_a(l(u))$. Hence $(l(u),l(ua))\in R_a$. This implies that  $\mu(l(u)a)=
\mu(l(ua))$. 

The length of $l(u)$ is clearly $\leq N^n$. 

Denote by $t(u)$ the time needed to compute $l(u)$. We show that it is $\leq C\frac{N^{|u|}-1}{N-1}$, 
which is exponential. This is true for $|u|=1$ since $t(a)=0$. Assume that the inequality is true for $u$. 
Then $t(ua)$ is $\leq$ the time to compute $l(u)$, plus the time needed to compute $l(ua)=\tau_a(l(u))$ 
from $l(u)$ (which is $\leq C|l(u)|
$); thus $t(ua)\leq C\frac{N^{|u|}-1}{N-1}+CN^{|u|}=C\frac{N^{|ua|}-1}{N-1}$.

To show that $l(u)$ is a representative of $u$ we proceed by induction. By construction, $l(a)\in L$ and $\mu(a)=\mu(l(a))$. Now assume $l(u)\in L$ and $\mu(u)=\mu(l(u))$.  Since $(l(u),l(ua))\in R_a$, we have $l(ua)\in L$ and since $\mu(l(u)a)=\mu(l(ua))$, 
we have $\mu(l(ua))=\mu(l(u))\mu(a)=\mu(u)\mu(a)=\mu(ua)$.
\end{proof}

\subsection{Presentation}

Recall that a semigroup $S$ is {\em rationally presented} if it has a finite generating set $A$ and a 
presentation $\langle A,R\rangle$ where $R$ is a rational subset of $A ^+\times A^+$. 

\begin{theorem}\label{present}
If $S$ is a quasi-automatic semigroup, then it is rationally presented. 
\end{theorem}

\begin{proof}
Consider a quasi-automatic structure $(A,\mu, L,R,(R_a)_{a\in A})$ on $S$.

We use the function $l$ of Theorem \ref{repr}.
Consider the semigroup 
congruence $\equiv$ generated by the relations determined by the pairs in $T=\cup_{a\in A}\{(ua,v)|(u,v)\in R_a\}\cup R\cup\{(a,l(a))|a\in A\}$. Since $\{(ua,v)|(u,v)\in R_a\}=R_a(a,1)$ is rational, $T$ is a finite union of rational relations, hence is rational.

We contend that for any words $u,v\in A^+$, $u\equiv v$ if and only if $\mu(u)=\mu(v)$. This will imply the theorem.

By construction, if $(u,v)\in T$, then $\mu(u)=\mu(v)$. Since $T$ generates $\equiv$, we obtain that $u\equiv v$ implies $\mu(u)=\mu(v)$.

Conversely, let $u,v$ be such that $\mu(u)=\mu(v)$. We claim that for any word $w\in A^+$, $l(w)\equiv w
$.The claim implies that
$l(u)\equiv u$, $l(v)\equiv v$. By what we have already proved, we have $\mu(l(u))=\mu(u)$ and 
$\mu(l(v))=\mu(v)$. Thus $\mu(l(u))=\mu(l(v))$. Since both words are in $L$, we obtain by definition of $R$ that $(l(u),l(v))\in R$, hence $\in T$ and 
therefore $l(u)\equiv l(v)$. It follows from the claim that $u\equiv v$.

It remains to prove the claim. It is true if $w=a\in A$, since $(a,l(a))\in T$. Suppose now that $l(w)\equiv w$. We show that $l(wa)\equiv wa$. Since $(l(w),l(wa))\in R_a$ by Theorem \ref{repr}, we have $(l(w)a,l(wa))\in T$, and therefore $l(w)a\equiv l(wa)$. Thus $wa \equiv l(w)a \equiv  l(wa)$. This proves the claim by induction.
\end{proof}

\begin{corollary}
Each quasi-automatic semigroup $S$ has a rational presentation $<A,T>$ such that for any words $u,v\in A^+
$, $u=v$ in $S$ if and only if, for $n=|u|+|v|+1$ and for some words $w_0,w_1,\ldots,w_n$, one has 
$w_0=u$, $w_n=v$ and each $w_{i+1}$ is obtained from $w_i$ by replacing some prefix $x$ of $w_i$ by 
some word $y$, with $(x,y)$ or $(y,x)\in T$. Moreover the lengths of the words $w_i$ are exponentially bounded with respect to $n$.
\end{corollary}

\begin{proof} We take the same $T$  as in the previous proof. Then the  "if" part is evident. In order to prove the "only if" part, we follow the previous proof.

We take $u=a_1\cdots a_k$, $v=b_1\cdots b_l$, ($a_i,b_j\in A$), $n=k+l+1$, $u=p_is_i$, $p_i$ of length $i$, $v=p'_js'_j$, $p'_j$ of length $j$, $w_i=l(p_i)s_i$, $i=1,\ldots,k$, $w_j=l(p'_{n-j})s'_{n-j}$, $j=k+1,\cdots,n-1$.

We have $(w_0,w_1)=(a_1s_1,l(a_1)s_1)$ and $(a_1,l(a_1))\in T$. 

Moreover, for $i=1,\ldots,k-1$, $(w_i,w_{i+1})=(l(p_i)a_{i+1}s_{i+1},l(p_ia_{i+1}),s_{i+1})$ and $(l(p_i)a_{i
+1},l(p_ia_{i+1}))\in T$ since $(l(p_i),l(p_ia_{i+1})) \in R_{a_{i+1}}$.

Note that $w_k=l(u)$ and $w_{k+1}=l(v)$. Thus $(w_k,w_{k+1})\in R\subset T$.

The rest of the argument is similar.
\end{proof}

\subsection{Word problem}

By definition, $S$ is isomorphic to the quotient 
of the free semigroup $A^{+}$ by the congruence 
generated by the pairs $(u,v)\in R$. A similar definition
holds for monoid presentations where $R\subseteq A^{*}\times A^{*}$
and 
$A^{*}$ are  substituted for $R\subseteq A^{+}\times A^{+}$ and $A^{+}$.

We recall that the word problem for a presentation consists 
of determining whether or not two words $u$ and $v$ are equivalent.

\begin{theorem}\label{word}
If $S$ is a quasi-automatic semigroup, then the word problem in $S$ is decidable in exponential time.
\end{theorem}

\begin{proof} 

The algorithm for the word problem goes as follows: let $u,v$ be two words; compute $l(u)$ and $l(v)$; check if $(l(u),l(v))$ is in $R$. If 
yes, then $u=v$ in $S$; if no, $u\neq v$ in $S$.

Regarding complexity, we may by Theorem \ref{repr} compute $l(u)$ and $l(v)$ in exponential time with 
respect to $n=max(|u|,|v|)$. Moreover their lengths are at most exponential in $n$. In order to conclude, we apply the following result: given a rational relation $T$, and two words $x,y$, one may check if $(x,y)\in T$, in quadratic time with respect to $max(|x|,|y|)$, see \cite{LN} Theorem 3.3.
\end{proof}

\subsection{Weak Lipschitz property}\label{Lip}

Let $S$ be a semigroup with generating set $A$. The {\em distance} between two elements in $S$ is the 
distance between them in the corresponding Cayley graph, viewed as an undirected graph.
Moreover, let $\mu$ as before and $L\subset A^+$ some language satisfying $\mu(L)=S$. Suppose that 
for some $P$, and for any words $u,v$ in $L$, such that the distance of $\mu(u)$ and $\mu(v)$ is at most 
1, one 
has: there exist $n$ and $a_1,\ldots,a_n,b_1,\ldots,b_n \in A\cup \{1\}$, such that in 
$A^*$ one has:
\begin{itemize}
\item $a_i=1\Leftrightarrow b_i\neq 1$;
\item $u=a_1\cdots a_n$, $v=b_1\cdots b_n$;
\item for any 
$i=0,\ldots,n$, the distance between $\mu(a_1...a_i)$ and $\mu(b_1\cdots b_i)$ is at most $P$.
\end{itemize}

In this case, we say that the triple $(S,A,L)$ has the {\em weak Lipschitz property}. 

The first condition is useful for the proof and for applications. It is however not essential: it may be skipped and then the new property is 
equivalent to the previous one. 

The weak Lipschitz property 
implies the {\em undirected asynchronous fellow traveler's property} of \cite{WWD}, Definition 2.7. 

The weak Lipschitz property is a  weak form of the {\em Lipschitz property} (\cite{ECHLPT} Lemma 2.3.2), 
or {\em fellow traveler 
property} (\cite{CRRT} Definition 3.11): in the latter, the prefixes of the same length of $u$ and $v$ are at 
bounded distance in $S$. 

Note that if $S$ has a zero, then the Lipschitz property (weak or not) is vacuous, since, as observed in 
\cite{CRRT} p. 375, the distance between any two elements in $S$ is bounded: it is at most twice the length of the zero. This implies that in 
general, the converse of Theorem \ref{fellow} does not hold, see \cite{CRRT} p. 375. However for groups, automacity is 
equivalent to the Lipschitz property, see \cite{ECHLPT} Theorem 2.3.5. Also, for automatic semigroups, 
there is a geometric characterization, see \cite{HT} Theorem 4.8; see also \cite{SS} for a geometric 
characterization of a stronger version of automatic monoids.

\begin{theorem}\label{fellow}
Let $S$ be a semigroup with a finite generating set $A$.
Let $(A,\mu, L,R,(R_a)_{a\in A})$ be a quasi-automatic structure on $S$. 
Then $(S,A,L)$ has the weak Lipschitz property.
\end{theorem}


\begin{lemma}
Let $H$ be a rational subset of $A^*$. Let  $M$ be the number of states of some automaton recognizing $H$. Then for any $w\in H$ and for 
any prefix $w_1$ of $w$, there exists a word $m$ of length at most $M$ such that $w_1m\in H$.
\end{lemma}

The proof of this lemma is a straightforward exercise in automata theory. 

\begin{proof} [Proof of Theorem \ref{fellow}]
By Nivat's theorem, for each rational relation $T$, there exists a rational subset $H$ of some finitely 
generated free monoid $B^*$ and 
alphabetic homomorphisms $\alpha, \beta:B^*\rightarrow A^*$ such that $T=\{(\alpha(w),\beta(w))|w\in H\}
$. Moreover, for any $b\in B$, $\alpha(b)=1$ if and only if $\beta(b)\neq 1$. Note that this ensures that $|m|=|\alpha(m)|+|\beta(m)|$ for any word in $m\in B^*$.

We apply this theorem to $T=R$ and $T=R_a$, $a\in A$, or the inverses of them; we then take $N$ to be the 
maximum of the corresponding constants $M$ given in the previous lemma. We take $P=N+1$.

Let $u,v$ be such that $\mu(u),\mu(v)$ are at distance at most 1. Then $(u,v)\in T$, where $T$ is one of the relations above. There exists $w
\in H$ such that $(u,v)=(\alpha(w),\beta(w))$. Let $w=c_1\ldots c_n$, $c_i\in B$. By the properties of $\alpha,\beta$, $n=|u|+|v|$. Let $\alpha(c_i)=a_i$, $\beta(c_i)=b_i$; then $a_i,b_i\in A\cup \{1\}$.

Let $i=0,\ldots,n$. By the previous lemma, for some word $m\in B^*$ of length at most $N$, we have 
$w'=c_1\cdots c_im\in H$ 
and therefore $(\alpha(w'),\beta(w'))\in T$. Then the distance of the images under $\mu$ of $\alpha(w')$ 
and $\beta(w')$ is at most 1. Since $\alpha(w')=a_1\cdots a_i\alpha(m)$, and $\beta(w')=b_1\cdots b_i
\beta(m)$, the distance between $\mu(a_1\cdots a_i)$ and $\mu(b_1\cdots b_i)$ is at most $|\alpha(m)| 
+1+|\beta(m)|$. This is equal to $|m|+1\leq N+1=P$.
\end{proof}

\subsection{Graded quasi-automatic semigroups}

We say that a semigroup $S$ is {\em graded} if it has a {\em degree}, that is a semigroup homomorphism $S$ into $(\mathbb N,+)$, and if it is generated as semigroup by elements of degree 1. There are many graded semigroups: $A^*$, $\mathbb N^k$, plactic monoids, braid monoids \ldots, and more generally each semigroup having a homogeneous presentation.

\begin{theorem} \label{graded}Let $S$ be a graded semigroup. If $S$ is quasi-automatic, then it is automatic.
\end{theorem}

\begin{proof}

Recall from Section \ref{rat} the identification of any subset $T$ of pairs of words of equal length in $A^*$ with a subset of $(A\times A)^*$.
Recall also the following: let $T$ be a rational relation such that if $(u,v)\in T$, then $u,v$ have same length. Then $T$ is rational as subset of $(A\times A)^*$.

Since $S$ is graded, it has a generating subset $A'$, each element of which is of degree 1; since $S$ is finitely generated, some finite subset $A$ of $A'$ generates $S$. Let $(A,\mu, L, R, (R_a)_{a\in A})$ be a quasi-automatic structure on $S$. Since the generators are of degree 1, $\mu$ preserves the degree. 

By the property (2) of $R$, we must have $|u|=|v|$ for any $(u,v)\in R$; thus, by Eilenberg's theorem, $R$ is a rational subset of $(A\times A)^*$. 

Similarly, by property (3), for each $(u,v)\in R_a$, $|u|+1=|v|$; let $\$$ be a new symbol. The set $\{(u\$,v)|(u,v)\in R_a\}$ is clearly rational. Thus by Eilenberg's theorem, it is a rational subset of $(A\times A)^*$.

It follows that $R,R_a$ form an automatic semigroup structure in the sense of \cite{CRRT}, see also Section \ref{autom}.
\end{proof}

\subsection{Groups}

We have seen that a quasi-automatic semigroup is rationally presented. By a theorem of Anisimov and Seifert \cite{AS} (see also \cite{B} Theorem III.2.7), every subgroup of the free group which happens to be a rational subset of the free group is actually finitely generated. Hence every rationally presented group is actually finitely presented.
For a quasi-automatic group, this will be proved below, with the further property that the group 
has an exponential {\em isoperimetric inequality}.

An isoperimetric inequality for a group $G$ means that the group is of the form $F(A)/N$, where $A$ is 
finite,  $F(A)$ is the free group on $A$, $N$ is a normal subgroup of $F(A)$ generated, as normal 
subgroup, by a finite set $E$ and that for some function $f$
and any $g$ in $N$, 
$g$ is a product of no more than $f(|g|)$ elements of the form 
$uvu^{-1}$, $u\in F(A)$, $v\in E$ (here $|g|$ is the length of the reduced word representing $g$). Of course, the inequality is called exponential if $f$ grows exponentially.

Isoperimetric inequalities are important since they characterize finitely presented groups having a 
decidable word problem: the function $f$ must be a recursive function, see \cite{ECHLPT}, Theorem 
2.2.5. The terminology comes from the fact that the length of $g$ is the perimeter of a Dehn diagram for 
$g$, whereas the number of cells is $f(|g|)$, see
Section 2.2 and in particular p. 44 in \cite{ECHLPT}. 

\begin{theorem}\label{isop}
Let $G$ be a group which is a quasi-automatic semigroup. Then $G$ has a finite presentation, as group, with an exponential isoperimetric inequality.
\end{theorem}

In order to prove this result, we follow as much as possible the proof of Theorem 2.3.12 (due to Thurston) in \cite{E}, which asserts that an automatic group is finitely presented, with a quadratic isoperimetric inequality. 

We need some notations. We may assume that $S$ has a quasi-automatic semigroup structure, with respect to a finite set $B=A\sqcup A^{-1}$ of generators {\em closed under inversion}: this means that there is an anti-automorphism of the free monoid $B^*$, denoted $w\mapsto w^{-1}$, such that for any word $\mu(w^{-1})=\mu(w)^{-1}$ (we extend $\mu$ to $B^*$ by $\mu(1)=1$). 

The homomorphism $\mu$ factorizes as $\nu\circ\pi$, with $
\pi:B^*\rightarrow F(A)$ the natural homomorphism commuting with inversion, and $\nu$ a surjective group homomorphism $F(A)\rightarrow G$.

We begin by a lemma.

\begin{lemma} 
Let $n\geq 1$, and $g_0,\ldots,g_n,a_1,\ldots,a_n, b_1,\ldots, b_n$ be elements of a group. Let $h_i=g_0b_1\cdots b_{i-1}g_{i-1}^{-1}$ ($i=1,\ldots,n$).
Then
$$
\prod_{1\leq i\leq n} h_i(a_ig_ib_i^{-1}g_{i-1}^{-1})h_i^{-1}=a_1\cdots a_ng_nb_n^{-1}\cdots b_1^{-1}g_0^{-1}.
$$
\end{lemma}

\begin{proof}
If $n=1$, this is clear since $h_1=1$. Suppose that it is true for $n$ and prove it for $n+1$. We have $\prod_{1\leq i\leq n
+1} h_i(a_ig_ib_i^{-1}g_{i-1}^{-1})h_i^{-1}=\prod_{1\leq i\leq n} h_i(a_ig_ib_i^{-1}g_{i-1}^{-1})h_i^{-1}\times 
h_{n+1}(a_{n+1}g_{n+1}b_{n+1}^{-1}g_{n}^{-1})h_{n+1}^{-1}$. This is equal by induction to 
$a_1\cdots a_ng_nb_n^{-1}\cdots b_1^{-1}g_0^{-1}h_{n+1}(a_{n+1}g_{n+1}b_{n+1}^{-1}g_{n}^{-1})h_{n
+1}^{-1}=a_1\cdots a_ng_nb_n^{-1}\cdots b_1^{-1}g_0^{-1}g_0b_1\cdots b_{n}g_{n}^{-1}(a_{n+1}g_{n+1}
b_{n+1}^{-1}g_{n}^{-1})g_nb_n^{-1}\cdots b_1^{-1}g_0^{-1}=a_1\cdots a_{n+1}g_{n+1}b_{n+1}^{-1}\cdots 
b_1^{-1}g_0^{-1}$.
\end{proof}

\begin{lemma} Take the notations above.
Let  $u,v$ in $L$ be such that the distance of $\mu(u)$ and $\mu(v)$ is at most 1.
Let $\mu(v)=\mu(u)
\mu(c)$, $c\in B\cup \{1\}$. Then $\pi(u)\pi(c)\pi(v)^{-1}$ is a product of no more that $|u|+|v|$ elements of the form $\pi(h)\pi(g)\pi(h)^{-1}$, with $g,h\in B^*$, $\mu(g)=1$ and $g$ of length at most $2P+2$.
\end{lemma}

\begin{proof} We apply Theorem \ref{fellow}, and take the notations given in the definition of the weak Lipschitz property. Since $\mu(a_1\cdots a_i)$ and $\mu(b_1\cdots b_i)$ are at distance at most $P$, we may find $g_i$ in $B^*$ of 
length at most $P$ such that for any $i=0,\ldots,n$, $\mu(a_1\cdots a_i)\mu(g_i)=\mu(b_1\cdots b_i)$. We may assume that $g_0=1$ and $g_n=c$. 

Then, for $i\geq 1$, $\mu(a_1\cdots a_{i-1})\mu(a_ig_ib_i^{-1}g_{i-1}^{-1})=\mu(b_1\cdots b_i)
\mu(b_i^{-1}g_{i-1}^{-1})=\mu(a_1\cdots a_{i-1})$. Thus $\mu(a_ig_ib_i^{-1}g_{i-1}^{-1})=1$. The 
word 
$a_ig_ib_i^{-1}g_{i-1}^{-1}$ is of length $\leq 2P+2$.

We have $\pi(u)\pi(c)\pi(v)^{-1}=\pi(a_1)\cdots \pi(a_n)\pi(g_n)\pi(b_n)^{-1}\cdots 
\pi(b_1)^{-1}\pi(g_0)^{-1}$. To conclude, we apply the previous lemma, with $a_i$ replaced by $\pi(a_i)$ 
and so on.
\end{proof}

\begin{proof} [Proof of Theorem \ref{isop}]
We have $G=F(A)/N$, $N=Ker(\nu)$. We show below that $N$ is generated, as normal subgroup, by the 
finitely many elements $\pi(g)$, with $g\in \mu^{-1}(1)$ of length at most $2P+2$, and $P$ the constant of Theorem 
\ref{fellow}. Moreover, we show that each element $z$ in $N$ of reduced length $k$ is a product of no 
more than 
an exponential function of $k$ elements of the form $yxy^{-1}$, $y\in F(A)$, $x$ a generator of $N$. This 
will prove the theorem.

Since $N=\pi(\mu^{-1}(1))$, we may write $z=\pi(w)$, $w\in B^*$ of length $k$ and such that $\mu(w)=1$. 
Let $p_t$ be the prefix of length $t$ of $w$. Let $u_t=l(p_t)\in L$ for $t=0,\cdots,k-1$; let $u_k=u_0$. Note that $u_0=l(1)$ is independent of $w$ and we may assume that $|u_0|\leq 2P+2$ (by 
taking a larger $P$ if necessary). Then 
by the property of $l$, see Theorem \ref{repr}, $\mu(p_t)=\mu(u_t)$ for all $t$; for $t=k$, we have 
also $\mu(p_k)=\mu(w)=1=\mu(p_0)=\mu(u_0)=\mu(u_k)$.

By Theorem \ref{repr}, $|u_t|\leq N^t\leq N^k$. Let $w=c_1\cdots c_k$, $c_i\in B$. Note that, since $p_{t
+1}=p_tc_{t+1}$, $\mu(u_{t+1})=\mu(u_tc_{t+1})$. 

Applying the previous lemma to $u_t$ and $u_{t+1}$, we see that $\pi(u_t)\pi(c_{t+1})\pi(u_{t+1})^{-1}$ is a product 
of at most $|u_t|+|u_{t+1}|$ elements of the form $\pi(h)\pi(g)\pi(h)^{-1}$, with $g,h\in B^*$, $
\mu(g)=1$ and $g$ of length at most $2P+2$. Note that $|u_t|+|u_{t+1}|\leq 2L^k$.

Now the product of all $\pi(u_t)\pi(c_{t+1})\pi(u_{t+1})^{-1}$, $t=0,\ldots, k-1$, is equal to $\pi(u_0)\pi(w)
\pi(u_k)^{-1}=\pi(u_0)z\pi(u_0)^{-1}$. Recall that $\pi(u_0)$ is in $N$ and of length bounded by $2P+2$. Thus $x$ is a product of no more that $2+2kL^k$ elements of the form $y
\pi(g)y^{-1}$, with $y\in F(A)$, $g\in B^*$, $\mu(g)=1$ and $g$ of length at most $2P+2$. 
\end{proof}

We have mentioned before that the Lipschitz property for groups implies automaticity. For the weak property, and quasi-automaticity, we have the similar result.

\begin{theorem} Let $G$ be a group, $B=A\sqcup A^{-1}$ an alphabet closed under inversion, 
 $\mu:B^+\rightarrow G$ the natural homomorphism and $L\subset A^+$ a rational language such that $
 \mu(L)=G$. If $(G,B,L)$ has the weak Lipschitz property, then $G$ is quasi-automatic.
\end{theorem}

\begin{proof}
Let $P$ be as in the definition of the weak Lipschitz property (see the beginning of Section \ref{Lip}). Let $
(Q,i,F)$ be a finite deterministic automaton recognizing $L$, with transitions denoted by $qa$, if $q\in Q$, 
$a\in B$; more generally, $qw$ denotes the state reached from state $q$ after having read word $w$. Let $G_0$ denote the set of elements of $G$ of length $\leq P$ (that is, at distance at most $P$ 
from $1$ in the Cayley graph); $G_0$ is finite. 

We construct transducers $T_a$, $a\in B\cup\{1\}$, which will behave all the same, except for the final 
states. Their set of states is $Q^2\times G_0$. The initial state is $(i,i,1)$ and the final states are all $
(p,q,a)$ for all possible final states $p,q\in F$. There is a transition $(p,q,g)\rightarrow (p',q',g')$, labelled 
$(a,b)$, with $a,b\in B\cup\{1\}$ and exactly one of $a$ or $b$ equal to $1$, if and only 
if: $pa=p',qb=q'$ and $g'=\mu(a)^{-1}g\mu(b)$.

It follows easily from this definition that if in $T_a$, there is a path $(p,q,g)\rightarrow (p',q',g')$ labelled $(u,v)$, then $g'=\mu(u)^{-1}g\mu(v)$.

We verify that $T_a$ recognizes $R_a$. Let $(u,v)$ be recognized by $T_a$. Let $ (i,i,1)\rightarrow 
(p',q',a)$ be a successful path with label $(u,v)$. Then $iu=p'\in F$, $iv=q'\in F$, $a=\mu(u)^{-1}\mu(v)\in G_0$. 
Hence $u,v\in L$ and $\mu(u)a=\mu(v)$; thus $(u,v)\in R_a$.

Conversely, suppose that $(u,v)\in R_a$. By the Lipschitz property, we may 
write $u=a_1\ldots a_n$, $v=b_1\ldots b_n$, $n=|u|+|v|$, and for any $i$ exactly one of $a_i$ or $b_i$ is 
equal to $1$; moreover, for any $i$, the distance from $\mu(a_1\cdots a_i)$ to $\mu(b_1\cdots b_i)$ is 
at most $P$. Thus there exist $g_i\in G_0$ such that $\mu(a_1\cdots a_i)g_i=\mu(b_1\cdots b_i)$. Note that, since $\mu(u)a=\mu(v)$, we have $g_n=a$. Let $p_0=i, p_1,p_2,\ldots,p_n$ (resp. $q_0=i, q_1,q_2,\ldots,q_n$) be the states of the path in the automaton $(Q,i,F)$ labelled $u$. In particular, $p_n,q_n\in F$ since $u,v\in L$. It 
follows that we have in $T_a$ a path $(i,i,1)=(p_0,q_0,g_0)\rightarrow (p_1,q_1,g_1) \rightarrow 
(p_2,q_2,g_2) \cdots \rightarrow (p_n,q_n,g_n)$, whose transitions are labelled $(a_1,b_1)$, $(a_2,b_2)$, $
\cdots(a_n,b_n)$ (with $g_0=1$): indeed, for any $i\geq 1$, we have $g_i=\mu(a_i)^{-1}g_{i-1}\mu(b_i)$. Since $g_n=a$, the path is successful in $T_a$, which therefore recognizes $(u,v)$.
\end{proof}

\section{Some open questions}

1. We have seen that asynchronous automatic semigroups are quasi-automatic. But we have no example of a semigroup that is quasi-automatic, but not asynchronous automatic. Note that it is known that deterministic 2-tape automata do not recognize all rational relations, see \cite{FR} Lemma 3.

2. It is known that if a group is bi-automatic (that is, the group and the opposite one are automatic), then 
the conjugation problem is decidable (see \cite{ECHLPT} Theorem 2.5.7). We do not know if the 
conjugation problem for bi-quasi-automatic groups is decidable. A first attempt to mimick the proof in 
\cite{ECHLPT} leads to undecidable properties of rational relations, see \cite{B} Theorem III.8.4.

3. It is known that for automatic groups (synchronous or asynchronous), one may find a rational set of unique representatives, see \cite{ECHLPT} Theorems 2.5.1 and 7.3.2. We do not 
know if this is true for quasi-automatic semigroups. A naive attempt to prove this leads to intersection of 
rational relations, which are not rational in general, see \cite{B} Example III.2.5. Note that it is conjectured  
in \cite{J} that if a rational relation is an equivalence relation, then it has a rational cross-section (it is 
proved for deterministic relations). This 
would give an affirmative answer to the previous question.

4. Is it decidable if a quasi-automatic semigroup is automatic (synchronous or asynchronous)? The similar 
problem for rational relations is undecidable, see \cite{B} III.8.4.

%
%
%

5. Given a quasi-automatic semigroup $S$, we do not know 
whether or not the following problems are decidable: is $S$ a monoid?
is $S$ finite?

Concerning  the first question, assume we are given representatives   $l_{a}\in L$
for all $a\in A$. 
The existence of  left neutral elements is decidable. Indeed, for every word $l\in L$,
its image is a left neutral element if and only if $(l,l_{a})\in R_{a}$. Therefore,
 denoting by $\pi$ the projection of $A^*\times A^*$ onto the first component,
the image of $l\in L$  by $\mu$ is  a left neutral element, if 
and only if it belongs to the   intersection of the following  (finitely many) rational subsets
of $A^*$
$$
E=\bigcap_{a\in A} 
\pi((A^*\times \{l_{a}\})\cap R_{a})
$$
which can be effectively computed. Now, we prove that the existence of 
a neutral element is semi-decidable. Indeed, because of the above discussion
it suffices to find some left neutral element which is also
a right neutral element. This can be done as follows: enumerate all elements 
of $E$ and for each $e\in E$ check that $(\bigcup_{a\in A}  (l_a,l_a))\subseteq  R_{e}$
for all $a\in A$. If the semigroup has a neutral element the procedure will find it, 
otherwise it will keep computing unless $E$ is finite. 

The second problem is also semi-decidable. Indeed, we can compute a 
representative $l({w})\in L$ of every $w\in A^*$, see Theorem \ref{repr}. 
Then the semigroup is finite if and only if there exists an integer $n$ such that 
for all $l({w})$ with $|w|=n+1$ there exists some $l({u})$ with $u<n$ such that 
$(l(u),l(w))\in R$.

Note that if $S$ is an automatic semigroup, then it has a language of unique 
representatives (see \cite{CRRT} Corollary 5.6), so that finiteness is evidently decidable.

 \end{document}